\documentclass[12pt]{amsart}
\usepackage{latexsym}
\usepackage{amssymb}
\usepackage{xcolor}
\numberwithin{equation}{section}
\makeatletter
\renewcommand{\subsection}{\@startsection
{subsection}{2}{0mm}{\baselineskip}{-0.25cm}
{\normalfont\normalsize\bf}}
\makeatother

\newtheorem{theorem}{Theorem}[section]
\newtheorem{proposition}[theorem]{Proposition}
\newtheorem{lemma}[theorem]{Lemma}
\newtheorem{corollary}[theorem]{Corollary}

\newtheorem{remark}[theorem]{Remark}

{\theoremstyle{definition}
\newtheorem*{definition*}{Definition}
}

\newtheorem*{proposition*}{Proposition}
\newtheorem*{corollary*}{Corollary}
\newtheorem*{lemma*}{Lemma}
\theoremstyle{remark}
\def\qb{{\bar q}}

\def\P{\mathbf P}

\def\cA{\mathcal A}

\def\cB{\mathcal B}
\def\cC{\mathcal C}
\def\cD{\mathcal D}

\def\cQ{\mathcal Q}
\def\bfq{{\bar{\mathbb F}}_q}

\def\cP{\mathcal P}
\def\cX{\mathcal X}

\def\Aut{{\rm Aut}}

\def\deg{{\rm deg}}

\def\min{{\rm min}}

\def\fq{\mathbb F_q}

\begin{document}

\title{On AG codes from a generalization of the Deligne-Lustzig curve of Suzuki type}

\author{Marco Timpanella}

\begin{abstract}
In this paper, Algebraic-Geometric (AG) codes and quantum codes associated to a family of curves which comprises the famous Suzuki curve are investigated. The Weierstrass semigroup at some rational point is computed. Notably, each curve in the family turn out to be a Castle curve over some finite field, and a weak Castle curve over its extensions. This is a relevant feature when codes constructed from the curve are considered.
\end{abstract}

\maketitle

\section{Introduction}
Let $\mathbb{F}_{q}$ be a finite field with $q$ elements, where $q$ is a power of a prime $p$. An algebraic curve over $\mathbb{F}_{q}$ is a projective, absolutely irreducible, non-singular, algebraic variety of dimension $1$ defined over $\mathbb{F}_{q}$.
Among algebraic curves over finite fields (i.e. projective, absolutely irreducible, non-singular, algebraic varieties of dimension $1$ defined over a finite field $\mathbb{F}_q$), a prominent role is played by the so called Deligne-Lusztig
curves associated with the Projective Special Unitary Group $\mathrm{PSU}(3,q)$, the Suzuki group, and the Ree group. In fact, these curves are exceptional both for being optimal with respect to the number of $\mathbb F_\ell$-rational points for some $\ell$ and for having a very large automorphism group with respect to their genus.

Curves possessing a large number of rational points hold significant interest both in their own right and for their applications in Coding Theory. Goppa's work \cite{Goppa1} introduced a fundamental concept: linear codes (the so-called Algebraic-Geometric codes) can be derived from an algebraic curve $\mathcal{X}$ defined over $\mathbb{F}_q$ by evaluating specific rational functions. These functions are chosen in such a way that their poles align with a given $\mathbb{F}_q$-rational divisor $G$, while the evaluation is performed at a distinct $\mathbb{F}_q$-rational divisor $D$ whose support is disjoint from that of $G$. AG codes are proven to have good performances provided that $\mathcal{X}$, $G$ and $D$ are carefully chosen in an appropriate way. In particular, as the relative Singleton defect of an AG code from a curve $\mathcal X$ is upper bounded by the ratio $g/N$, where $g$ is the genus of $\mathcal X$ and $N$ can be as large as the number of $\mathbb F_q$-rational points of $\mathcal X$, it follows that curves with many rational points with respect to their genus are of great interest in Coding Theory. In particular, AG codes from maximal curves (namely curves with the maximum possible number of rational points) have been widely investigated in the last years, see for instance \cite{KNT,LV,LT,MTZ}.

In this paper, we investigate a generalization of the Deligne-Lustzig curve of Suzuki type originally defined in \cite{RODI}, where it was noted that the number of automorphisms exceeds the Hurwitz bound. Our main original contribution is the investigation of the Weierstrass semigroup of the curve at a specific point, see Propositions \ref{semi} and \ref{simmetrico}, which leads to the proof that the curve is actually a Castle curve over $\mathbb F_q$ and a weak Castle curve over $\mathbb F_{q^i}$ for all $i\ge 1$. In this paper we also provide the proofs of some facts which are stated in \cite{RODI} without proofs. Both Castle and weak Castle curves are of particular interest in the context of applications of curves to linear codes. In fact, they combine the good properties of having a reasonable simple
handling and giving codes with excellent parameters. Also, these
codes have self-orthogonality properties which are very close to those
required for obtaining quantum stabilizer codes; see \cite{Castle1,Castle2,Castle3}. As an application of the curve being Castle, we provide a construction of quantum codes associated to the curve; see Proposition \ref{genquan} and the discussion at the end of Section \ref{ultima}.

\section{Background on algebraic curves and AG codes}

For a curve $\mathcal{X}$, we adopt the usual notation and terminology; see for instance \cite{HKT,stichtenoth1993}. In particular,  $\mathbb{F}_q(\mathcal{X})$ and $\mathcal{X}(\mathbb{F}_q)$ denote the field of $\mathbb{F}_q$-rational functions on $\cX$ and the set of $\mathbb{F}_q$-rational points of $\mathcal{X}$, respectively, and ${\rm{Div}}(\mathcal{X})$ denotes the set of divisors of $\mathcal{X}$, where a divisor $D\in {\rm{Div}}(\mathcal{X})$ is a formal sum $n_1P_1+\cdots+n_rP_r$, with $P_i \in \mathcal{X}$, $n_i \in \mathbb{Z}$ and $P_i\neq P_j$ if $i\neq j$.
 The support $\mbox{\rm Supp}(D)$ of the divisor $D$ is the set of points $P_i$ such that $n_i\neq 0$, while $\deg(D)=\sum_i n_i$ is the degree of $D$. 
The divisor $D$ is $\mathbb F_q$-rational if $n_i\neq 0$ implies $P_i\in \cX(\mathbb{F}_q)$.
For a function $f \in \mathbb{F}_q(\mathcal{X})$, $(f)$, $(f)_0$ and $(f)_{\infty}$ are the divisor of $f$, its divisor of zeroes and its divisor of poles, respectively.
The Weierstrass semigroup  $H(P)$ at $P\in \cX$ is $$H(P) := \{n \in \mathbb{N}_0 \ | \ \exists f \in \mathbb{F}_q(\mathcal{X}), (f)_{\infty}=nP\}= \{\rho_0=0<\rho_1<\rho_2<\cdots\}.$$
The Riemann-Roch space associated with an $\mathbb F_q$-rational divisor $D$ is
$$\mathcal{L}(D) := \{ f \in \mathcal{X}(\mathbb{F}_q) \ : \ (f)+D \geq 0\}\cup \{0\}$$
and its vector space dimension over $\mathbb{F}_q$ is $\ell(D)$.

Fix a set of pairwise distinct $\mathbb{F}_q$-rational points $\{P_1,\cdots,P_N\}$, and let $D=P_1+\cdots+P_N$. Take another $\mathbb F_q$-rational divisor $G$ whose support is disjoint from the support of $D$. The AG code $C(D,G)$ is the (linear) subspace of $\mathbb{F}_q^N$ which is defined as the image of the evaluation map $ev :  \mathcal{L}(G) \to \mathbb{F}_q^N$ given by $ev(f) = (f(P_1),f(P_2) ,\ldots,f(P_N))$. 
In particular $C(D,G)$ has length $N$. Moreover, if $N>\deg(G)$ then $ev$ is an embedding and $\ell(G)$ equals the dimension of $C(D,G)$. The minimum distance $d$ of $C(D,G)$, usually depends on the choice of $D$ and $G$. A lower bound for $d$ is $ d^*=N-\deg(G)$, where $d^*$ is called the Goppa designed minimum distance of $C(D,G)$. Furthermore, if $\deg(G)>2\mathfrak{g}-2$ then $k=\deg(G)-\mathfrak{g}+1$ by the Riemann-Roch Theorem; see \cite[Theorem 2.65]{HLP}.\\
The dual code $C^{\bot} (D,G)$ can be obtained in a similar way from the $\mathbb{F}_q(\cX)$-vector space $\Omega(\cX)$ of differential forms over $\cX$.
For a differential $\omega\in \Omega(\cX)$, there is associated a divisor $(\omega)$ of $\cX$, whose degree is $2\mathfrak{g}-2$. For an $\mathbb{F}_q$-rational divisor $D$,  $$\Omega(D):=\{\omega\in \Omega(\cX)\ :\ (\omega)\geq D\}\cup \{0\} $$ 
is a $\mathbb{F}_q$-vector space of rational differential forms over $\cX$.
Then the code $C^{\bot}(D,G)$ coincides with the (linear) subspace of $\mathbb{F}_q^N$  which is the image of the vector space $\Omega(G-D)$ under the linear map $res_D:\Omega(G-D)\mapsto\mathbb{F}_q^N$ given by $res_D(\omega)=(res_{P_1}(\omega),\dots,res_{P_N}(\omega))$, where $res_{P_i}(\omega)$ is the residue of $\omega$ at $P_i$. 
In particular, $C^{\bot}(D,G)$ is an AG code with dimension $k^{\bot}=N-k$ and minimum distance $d^{\bot}\geq \deg{(G)}-2\mathfrak{g}+2$.\\

In the case where $G=\alpha P$, $\alpha \in \mathbb{N}_0$, $P \in \mathcal{X}(\mathbb{F}_q)$, the AG code ${C} (D,G)$ is referred to as one-point AG code. For a Weierstrass semigroup $H(P)= \{\rho_0=0<\rho_1<\rho_2<\cdots\}$ and an integer $\ell \geq0$, the Feng-Rao function is 
$$\nu_\ell := | \{(i,j) \in \mathbb{N}_0^2 \ : \ \rho_i+\rho_j = \rho_{\ell+1}\}|.$$ 
Consider $${C}_{\ell}(P)= {C}^{\bot}(P_1+P_2+\cdots+P_N,\rho_{\ell}P),$$ with $P,P_1,\ldots,P_N$ pairwise distint points in $\mathcal{X}(\mathbb{F}_q)$. The number 
$$d_{ORD} ({C}_{\ell}(P)) := \min\{\nu_{m} \ : \ m \geq \ell\}$$
is a lower bound for the minimum distance $d({C}_{\ell}(P))$ of the code ${C}_{\ell}(P)$ which is called the order bound or the Feng-Rao designed minimum distance of ${C}_{\ell}(P)$; see \cite[Theorem 4.13]{HLP}.\\

\section{Preliminaries}
Throughout the paper,  $q=2^s$ and $q_0=2^h$ with $2h<s$. Also, $\qb=q/q_0$ and
$n_1:=\qb/q_0$.
Let $\cC$ be the plane curve defined over $\fq$ by the equation
$$
X^{q_0}(X^q+X)=Y^q+Y.
$$
Also, let
$$
v(X,Y):=Y^\qb+X^{\qb+1}\,,\qquad w(X,Y):=Y^\qb
X^{n_1-1}+v(X,Y)^{\qb}
$$

Note for $s$ odd and $2h+1=s$ the curve $\cC$ is the Deligne-Lusztig
curve of Suzuki type.

The condition $2h<s$ is motivated as follows. For $2h=s$, that is $q_0=\qb=\sqrt q$,  the curve
$X^{q_0}(X^q+X)=Y^q+Y$ is reducible, as
$X^{q_0}(X^q+X)+Y^q+Y=\Pi_{\alpha^{q_0}=\alpha}(X^{q_0+1}+Y^{q_0}+Y+\alpha)$.
For $2h>s$ the curve $X^{2^h}(X^q+X)=Y^q+Y$ is birationally
equivalent to $X'^{q/2^h}(X'^q+X')=Y^q+Y$ by setting
$X'=Y^{q/2^h}+X^{(q/2^h)+1}$.

\begin{proposition}[Proposition 6.7 in \cite{RODI}]\label{irreducible}
The curve $\cC$ is absolutely irreducible. Moreover, there exists
a unique place of $\fq(\cC)$ centered at $Y_\infty$, the infinite point of the $Y$-axis.
\end{proposition}
\begin{proof}
Let $\cX$ be any component of $\cC$, with equation $H(X,Y)=0$ for some irreducible factor $H(X,Y)$ of $X^{q_0}(X^q+X)-Y^q-Y$. Then $\fq(\cX)=\fq(x,y)$,
with
\begin{equation}\label{zero}
x^{q_0}(x^q+x)=y^q+y\,.
\end{equation}
From
$$
(y^q+y)x^{q_0(n_1-1)}=x^{q+n_1q_0}+x^{n_1q_0+1}
$$
it follows that
$$
(y^{\qb}x^{n_1-1}+x^{n_1+\qb})^{q_0}=x^{n_1q_0+1}+yx^{q_0(n_1-1)}\,,
$$
and hence
\begin{equation}\label{first}
(x^{n_1+\qb})^{q_0}+x^{n_1q_0+1}=(y^{\qb}x^{n_1-1})^{q_0}+yx^{q_0(n_1-1)}\,.
\end{equation}
Now define $v:=v(x,y)$ and $w:=w(x,y)$. Then by straightforward
computation
\begin{equation}\label{v}
v^q+v=x^\qb(x^q+x)\,.
\end{equation}
By (\ref{first}),
$$(y^{\qb}x^{n_1-1})^{q_0}+yx^{q_0(n_1-1)}=v^{\qb
q_0}+v$$ and hence
\begin{equation}\label{w}
w^{q_0}=yx^{q_0(n_1-1)}+v\,.
\end{equation}
Now, let $\cP$ be any place of $\fq(x,y)$ centered at $Y_\infty$.
Let $s:=v_\cP(x)$. By  (\ref{zero}),
\begin{equation}\label{val}
-q\le s <0\,.
\end{equation}
Note that showing $s=-q$ is enough to prove both the statements in
the proposition. By (\ref{zero}) and (\ref{v}) it follows that
\begin{equation}\label{val1}
-q_0\le s <0\,\quad s(q_0+q)=qv_\cP(y)\,,\quad
s(\qb+q)=qv_\cP(v)\,.
\end{equation}
In particular, $\qb$ divides $s$. By (\ref{w}),
$$
q_0v_\cP(w)\ge \text{min}\{v_\cP(y)+q_0(n_1-1)v_\cP(x),v_\cP(v)\}\,,
$$
that is
$$
q_0v_\cP(w)\ge
\text{min}\{s\frac{q_0+q}{q}+q_0(n_1-1)s,s\frac{\qb+q}{q}\}\,.
$$
Since $s<0$ and $n_1>1$,
$s\frac{q_0+q}{q}+q_0(n_1-1)s<s\frac{\qb+q}{q}$ holds, and hence
$$
q_0v_\cP(w)=s\frac{q_0+q}{q}+q_0(n_1-1)s\,.
$$
This implies that $q_0$ divides $\frac{s}{\qb}$, which together
with (\ref{val1}) yields $q\mid s$. Finally $s=-q$ follows from
(\ref{val}), and the proposition is proved.
\end{proof}
According to the proof of Proposition \ref{irreducible}, from now on, $x$ and $y$ denote the algebraic functions in
$\fq(\cC)$ such that $\fq(\cC)=\fq(x,y)$ with
$x^{q_0}(x^q+x)=y^q+y$. Moreover, let $\cP_\infty$ be the only
place of $\fq(\cC)$ centered at $Y_\infty$. Finally, we set
$v=v(x,y)$, $w=w(x,y)$. The following statement follows from the proof of Proposition \ref{irreducible}.
\begin{proposition}\label{scho} In $\fq(\cC)$,
\begin{itemize}
\item[{\rm (1)}] $v_{\cP_\infty}(x)=-q$,
$v_{\cP_\infty}(y)=-q_0-q$;

\item[{\rm (2)}] $v_{\cP_\infty}(v)=-\qb-q$;

\item[{\rm (3)}] $v_{\cP_\infty}(w)=-(q(n_1-1)+\qb+1)$;

\end{itemize}
\end{proposition}
\begin{corollary} The rational function $vx^{n_1-2}/w\in \fq(\cC)$
is a local parameter at $\cP_{\infty}$.
\end{corollary}

\begin{proposition}[Proposition 6.8 in \cite{RODI}] The genus of $\cC$ is $g_\cC=\frac{1}{2}\qb(q-1)$.
\end{proposition}
\begin{proof}
We are going to apply Hilbert's different formula to the
extension $\fq(\cC)/ \fq(x)$. It is easy to see that $\fq(\cC)/
\fq(x)$ is a Galois extension. Its Galois group $\Gamma(\fq(\cC)/
\fq(x))$ consists of the automorphisms $\delta_a$, where
\begin{equation}\label{delta}
 \delta_a \ := \left \{ \begin{array}{lll}
         x \mapsto x, \\
         y \mapsto y+a\,,
                         \end{array}
          \right.
\end{equation}
with $a$ ranging over $\fq$. Then the Hurwitz's genus formula gives
\begin{equation}\label{hur}
2g_\cC-2=-2q+d(\cP_\infty\mid \cQ)\,,
\end{equation}
$\cQ$ being the infinite place of $\fq(x)$. For $a\in \fq$ we
compute $v_{\cP_\infty}(\delta_a(t)-t)$, where $t=vx^{n_1-2}/w$.
By straightforward computation,
$$
\begin{array}{cl} \delta_a(t)-t &  =\displaystyle \frac{vx^{n_1-2}+a^\qb
x^{n_1-2}}{w+a^\qb x^{n_1-1}+a^{\qb^2}}-\frac{vx^{n_1-2}}{w}
\\
{}& =\displaystyle \frac{w(vx^{n_1-2}+a^\qb x^{n_1-2})-(w+a^\qb
x^{n_1-1}+a^{\qb^2})vx^{n_1-2}}{(w+a^\qb
x^{n_1-1}+a^{\qb^2})w} \\

{} &=\displaystyle \frac{a^\qb (wx^{n_1-2}-vx^{2n_1-3}-a^\qb
vx^{n_1-2})}{(w+a^\qb x^{n_1-1}+a^{\qb^2})w}\,.
\end{array}
$$

Taking into account Proposition \ref{scho}, it follows that for
$\delta_a\neq id$
$$v_{\cP_\infty}(\delta_a(t)-t)=\qb+2\,.$$ Hilbert's
different formula yields $d(\cP_\infty\mid \cQ)=(q-1)(\qb+2)$, and
hence by (\ref{hur}) the statement is proved.
\end{proof}

Let $f$ be the morphism $f:=\cC \rightarrow
\P^4({\bar{\mathbb{F}}}_q)$ with coordinate functions
 $$
f:=(f_0:f_1:f_2:f_3:f_4)\,, $$ such that $f_0:=1,\,f_1:=x,\,
f_2:=vx^{n_1-2}, \, f_3:=y,\, f_4:=w$. They are uniquely
determined by $f$ up to a proportionality factor in $\fq(\cC)$.
For each point $P\in\cC$, we have
$f(P)=((t^{-e_P}f_0)(P),\ldots,(t^{-e_P}f_4)(P))$ where $e_P=-{\rm
min}\{v_P(f_0),\ldots,v_P(f_4)\}$ for a local parameter $t$ of
$\cC$ at $P$. It turns out that $f(\cC)$ is a  curve
not contained in any hyperplane of $\P^4(\bfq)$. For a point $P\in
f(\cC)$, the intersection multiplicity of $f(\cC)$ with a
hyperplane $H$ of equation $a_0X_0+\ldots+a_4X_4=0$ is
$v_P(a_0f_0+\ldots+a_4f_4)+e_P$, and the intersection divisor
$f^{-1}(H)$ cut out on $f(\cX)$ by $H$ is defined to be
$f^{-1}(H)=(a_0f_0+\ldots+a_4f_4)+E$ with $E=\sum e_pP$. By
Proposition \ref{scho}, we have $v_{\cP_\infty}(f_1)=-q$,
$v_{\cP_\infty}(f_2)=-(q(n_1-1)+\qb)$, $v_{P_\infty}(f_3)=-q_0-q$,
$v_{\cP_\infty}(f_4)=-(q(n_1-1)+\qb+1)$. Then
$e_{\cP_\infty}=q(n_1-1)+\qb+1$, and the representative
$(f_0/f_4:f_1/f_4:f_2/f_4:f_3/f_4:1)$ of $f$ is defined on
$\cP_{\infty}$. Hence $f(\cP_{\infty})=(0:0:0:0:1)$. For a point
$\cP\in\cC$, an integer $j$ is called a Hermitian $P$-invariant
 if there exists a hyperplane intersecting $f(\cC)$
at $f(P)$ with multiplicity $j$. There are exactly five pairwise
distinct Hermitian $P$-invariants. Such integers arranged in
increasing order define the order sequence of $\cC$ at $P$.
\begin{proposition}
\label{closed} $f(\cC)$ is a non-singular model defined over $\fq$
of $\cC$.
\end{proposition}
\begin{proof} We show that $f$ is a closed embedding. By the above discussion,
$f$ is bijective and $f(\cC)$ has no singular point.
\end{proof}

\section{Some automorphisms of $\fq(\cC)$}
For $b,c,d \in \fq$ with $d\neq 0$, we define the following
automorphisms of $\fq(\cC)$:
$$
 \alpha_{b,c} \ := \left \{ \begin{array}{lll}
         x \mapsto x+b, \\
         y \mapsto y+b^{q_0}x+c;
                         \end{array}
          \right.
$$
\begin{equation}
\label{gamma}
 \beta_d \ := \left \{ \begin{array}{lll}
         x \mapsto dx, \\
         y \mapsto d^{q_0+1}y;
                         \end{array}
          \right.
\end{equation}
Note that $\alpha_{b,c}^2=\delta_{b^{q_0+1}}$, with
$\delta_{b^{q_0+1}}$ as in (\ref{delta}). Let $\cA$, $\cB$, $\cD$
be the following subgroups of $\Aut (\cC)$:
$$\cA:=\{\alpha_{b,c}\mid b,c \in \fq\}\,,\quad
\cB:=\{\beta_{d}\mid d \in \fq, d\neq 0\}\,,\quad
\cD:=\{\delta_{a}\mid a \in \fq\}\,.$$

Let $\Gamma$ be the automorphism group of $\fq(\cC)$ generated by
$\cA$ and $\cB$. 
The number of elements in $\Gamma$ is at least $q^2(q-1)$, that is
$\#\Aut(\fq(\cX))>84(g_\cC-1)$ apart from the case $q\le 16$. The
sets $\{\cP_\infty\}$ and $\{\cC(\fq)\}\setminus \{\cP_\infty\}$
are two shorts orbits of $\Aut(\fq(\cX))$. The former is a
non-tame orbit, while the latter is tame.

\section{Weierstrass semigroup}\label{Section:semi}

The aim of this section is to prove the following result.

\begin{proposition}\label{semi} The Weierstrass semigroup at $\mathcal P_\infty$ is $H(\mathcal P_\infty)=\langle q,q+q_0,q+\bar{q},q(n-1)+\bar{q}+1 \rangle$.
\end{proposition}

Let $A$ be the numerical semigroup generated by $\{q, q+q_0, q+\bar{q},q(n-1)+\bar{q}+1\}$. To prove $A=H(\mathcal P_\infty)$, we will make use of the following basic definitions and results from the theory of numerical semigroups.

\begin{definition*}
Let $S\subset \mathbb{N}$ be a numerical semigroup.
\begin{itemize}
\item The genus $g(S)$ of $S$ is the cardinality of the set $\mathbb{N}\setminus S$ (which, by definition, is finite);
\item The conductor $c(S)$ of $S$ is  $c(S)=1+{\rm max}\{x\in \mathbb{N}\setminus S\}$. Also, $S$ is symmetric if $c(S)=2g(S)$;
\item The multiplicity $m(S)$ of $S$ is  $m(S)={\rm min}\{x\in S\}$;
\item For a non-zero element $s\in S$, the Ap\'ery set of $s$ is $$Ap(S,s):=\{x \in S \,|\,x-s \not\in S\}.$$
\end{itemize}
\end{definition*}
Note that $Ap(S,m(S))$ provides a complete set of minimal representatives for the congruence classes of $\mathbb{Z}$ modulo $m(S)$. As a consequence, the semigroup can be also described as $S = \{t m(S) + x\,:\, t\geq 0 \text{ and } x\in Ap(S,m(S))\}$.
A strong connection between the Ap\'ery sets, the genus, and the conductor of a numerical semigroup is given by the following result.
\begin{proposition}\label{resApery}
Let $S$ be a numerical semigroup and $s$ a non-zero element of $S$. Then $|Ap(S,s)|=s$, 
\begin{equation}\label{genereAp}
g(S)=\frac{1}{s}\sum_{x\in Ap(S,s)} x- \frac{s-1}{2},
\end{equation}
and
\begin{equation}\label{conductorAp}
c(S)=1+{\rm max}\{x\in Ap(S,s)\}-s.
\end{equation}
\end{proposition}

Observe that if $\bar{S}\subset S$ is a complete set of representatives for the congruence classes of $\mathbb{Z}$ modulo $m(S)$ (not necessarily minimal), then 
\begin{equation}\label{eq:completeset}
g(S)\leq \frac{1}{m(S)}\sum_{x\in \bar{S}} x- \frac{m(S)-1}{2},
\end{equation}
and the equality holds if and only if $\bar{S}=Ap(S,m(S))$.

By Proposition \ref{scho}, $\{q, q+q_0, q+\bar{q},q(n-1)+\bar{q}+1\}\subseteq H(\mathcal P_\infty)$ and hence $A$ is contained in $H(\mathcal P_\infty)$. In particular, $g(A)\geq g(H(\mathcal P_\infty))=\frac{1}{2}\bar{q}(q-1)$. To prove the other inequality, we explicitly compute the Ap\'ery set $Ap(A,q)$. Note that  $q$ is the multiplicity of $A$.

\begin{proposition}\label{barA}
The set 
$$
\bar{A}:=\{t_1(q+q_0)+t_2(q+\bar{q})+t_3(q(n-1)+\bar{q}+1)\,:\, 0\leq t_1\leq n-1, 0\leq t_2\leq q_0-1, 0\leq t_3\leq q_0-1\}
$$
is a complete set of representatives for the congruence classes of $\mathbb{Z}$ modulo $q$.
\end{proposition}
\begin{proof}
Clearly the size of $\bar{A}$ is at most $nq_0^2=q$. To prove the claim, we show that if $\bar{a}$ and $\bar{a}^\prime$ are two distinct elements of $\bar{A}$, then $\bar{a}\not\equiv \bar{a}^\prime \pmod{q}$. Indeed, let 
\begin{eqnarray*}
\bar{a}&=&t_1(q+q_0)+t_2(q+\bar{q})+t_3(q(n-1)+\bar{q}+1),\\
\bar{a}^\prime&=&t_1^\prime(q+q_0)+t_2^\prime(q+\bar{q})+t_3^\prime(q(n-1)+\bar{q}+1),
\end{eqnarray*}
and assume $\bar{a}\equiv \bar{a}^\prime \pmod{q}$. As $q_0$ divides $q$, we have $\bar{a}\equiv \bar{a}^\prime \pmod{q_0}$ and hence $t_3 \equiv t_3^\prime \pmod{q_0}$. Since $t_3,t_3^\prime\in \{0,\ldots,q_0-1\}$, we obtain $t_3=t_3^\prime$. The same argument, replacing $q_0$ with $\bar{q}$, yields $\bar{a}\equiv \bar{a}^\prime \pmod{\bar{q}}$ and hence $t_1q_0\equiv t_1^\prime q_0 \pmod{\bar{q}}.$ Then, $\bar{q}=nq_0$ yields $t_1\equiv t_1^\prime \pmod{n}$, which, combined with $t_1,t_1^\prime\in \{0,\ldots,n-1\}$, gives $t_1=t_1^\prime$.
Finally, $t_2(q+\bar{q}) \equiv t_2^\prime(q+\bar{q}) \pmod{q}$ yields $t_2\equiv t_2^\prime \pmod{q_0}$ and so $t_2=t_2^\prime$. Therefore $\bar{a}=\bar{a}^\prime$, which completes the proof.
\end{proof}

We are now in position to prove Proposition \ref{semi}.

\begin{proposition}\label{proofsemigruppo}
$H(\mathcal{P}_\infty)=A$.
\end{proposition}
\begin{proof}
As we already observed, $A\subseteq H(\mathcal{P}_\infty)$ and hence $g(A)\geq g(H(\mathcal{P}_\infty))$. On the other hand, Proposition \ref{barA} together with Equation \ref{eq:completeset} yield
$$
g(A)\leq \frac{1}{q}\sum_{x\in \bar{A}} x-\frac{q-1}{2}.
$$
By straightforward computation:
\begin{eqnarray*}
\sum_{x\in \bar{A}} x&=&\sum_{t_1=0}^{n-1}\sum_{t_2=0}^{q_0-1}\sum_{t_3=0}^{q_0-1}(t_1(q+q_0)+t_2(q+\bar{q})+t_3(q(n-1)+\bar{q}+1))\\
&=&\sum_{t_1=0}^{n-1}\sum_{t_2=0}^{q_0-1}\left(t_1q_0(q+q_0)+t_2q_0(q+\bar{q})+\frac{q_0(q_0-1)}{2}(q(n-1)+\bar{q}+1)\right)\\
&=&\sum_{t_1=0}^{n-1}\left(t_1q_0^2(q+q_0)+\frac{q_0^2(q_0-1)}{2}(q+\bar{q})+\frac{q_0^2(q_0-1)}{2}(q(n-1)+\bar{q}+1)\right)\\
&=&\left(\frac{n(n-1)q_0^2}{2}(q+q_0)+\frac{nq_0^2(q_0-1)}{2}(q+\bar{q})+\frac{nq_0^2(q_0-1)}{2}(q(n-1)+\bar{q}+1)\right)\\
&=&\frac{nq_0^2}{2}(\bar{q}q-\bar{q}+q-1)=\frac{q}{2}(\bar{q}q-\bar{q}+q-1),
\end{eqnarray*}
whence
\begin{eqnarray*}
g(A)\leq \frac{1}{2}(\bar{q}q-\bar{q}+q-1)-\frac{q-1}{2}=\frac{\bar{q}(q-1)}{2}=g_C=g(H(\mathcal{P}_\infty)).
\end{eqnarray*}
Therefore $g(A)=g(H(\mathcal{P}_\infty))$ and $A=H(\mathcal{P}_\infty)$.
\end{proof}

\begin{remark}\label{remApery}
By the proof of Proposition \ref{proofsemigruppo}, $g(A)=\frac{1}{q}\sum_{x\in \bar{A}} x-\frac{q-1}{2}$. Therefore, $\bar{A}$ is exactly the Ap\'ery set $Ap(A,q)=Ap(H(\mathcal{P}_\infty),q)$.
\end{remark}

\begin{proposition}\label{simmetrico}
The Weierstrass semigroup at $\mathcal P_\infty$ is symmetric.
\end{proposition}
\begin{proof}
By Propositions \ref{resApery} and \ref{proofsemigruppo}, together with Remark \ref{remApery}, the conductor of $H(\mathcal{P}_\infty)$ is
$$
c(H(\mathcal{P}_\infty))=1+{\rm max}\{x\in \bar{A}\}-q,
$$
that is
\begin{eqnarray*}
c(H(\mathcal{P}_\infty))&=&1+(n-1)(q+q_0)+(q_0-1)(q+\bar{q})+(q_0-1)(q(n-1)+\bar{q}+1)-q\\
&=& nq_0q-\bar{q}=\bar{q}q-\bar{q}=2g_C=2g(H(\mathcal{P}_\infty)),
\end{eqnarray*}
whence the claim follows.
\end{proof}
\section{AG codes and AG quantum codes}

\subsection{Number of rational points}\label{numberrationalpoints}
By the non-singularity of any affine point in $\cC$ and by
Proposition \ref{irreducible} it follows that the number of
$\fq$-rational points of $\cC$ is $N_1(\cC)=q^2+1$. This means
that
$$\frac{N_1(\cC)}{g_\cC}=2q_0+\frac{2}{q}+\frac{4}{\qb(q-1)}>2q_0\,.$$

By Proposition \ref{semi}, the smallest positive non gap at the $\fq$-rational point $\mathcal P_\infty$ is $q$. Hence the curve is $\fq$-optimal with respect to the Lewittes bound \cite{Lewittes}.

Let $N_i(\cC)$ be the number of of ${\mathbb F}_{q^i}$-rational
points of $\cC$. By computer results we checked that:

 \begin{itemize}
 \item $q=16$, $q_0=2$, $g_\cC=60$: $N_3(\cC)=N_2(\cC)=N_1(\cC)=1+256$, $N_4(\cC)=65537=q^4+1$;
 \item $q=32$, $q_0=2$, $g_\cC=248$: $N_2(\cC)=N_1(\cC)=1+1024$,
 $N_3(\cC)=1+1024+3\times 1024\times 31=96257$.
\end{itemize}
Notice that in the second case, $N_3(\cC)$ exceeds $(1/\sqrt
2)(q^3+1+2g_\cC\sqrt{ q^3})$. So the curve has ``many'' rational
points over ${\mathbb F}_{q^3}$.

\subsection{Quantum codes and Castle property}\label{ultima}

Let $\mathbb{H}=(\mathbb{C}^q)^{\otimes n}=\mathbb{C}^q \otimes \cdots \otimes \mathbb{C}^q$ be a $q^n$-dimensional Hilbert space. Then the $q$-ary quantum code $C$ of length $n$ and dimension $k$ are the $q^k$-dimensional Hilbert subspace of $\mathbb{H}$. Such quantum codes are denoted by $[[n,k,d]]_q$, where $d$ is the minimum distance.
As in the ordinary case, $C$ can correct up to $\lfloor \frac{d-1}{2}\rfloor$ errors.
Moreover, the quantum version of the Singleton bound states that for a $[[n,k,d]]_q$-quantum code, $2d+k\leq 2+n$ holds.
Again, by analogy with the ordinary case, the quantum Singleton defect and the relative quantum Singleton defect are defined to be $\delta_Q:= n-k-2d+2$ and $\Delta_Q:=\frac{\delta_Q}{n}$, respectively. 

The CSS construction \cite{CS96,St96} showed that quantum codes can be derived from classical linear codes verifying
certain self-orthogonality properties.

\begin{lemma}{\rm{(CSS construction)}}\label{CSS}
Let $C_1$ and $C_2$ be linear codes with parameters $[n,k_1,d_1]_q$ and $[n,k_2,d_2]_q$, respectively, and assume that $C_1 \subset C_2$. Then there exists a $[[n,k_2-k_1,d]]_q$-quantum code with $$d=\min \{ w(c)\, \vert \, c\in (C_2 \setminus C_1)\cup (C_1^{\perp} \setminus C_2^{\perp}) \}.$$
\end{lemma}
Among all the classical codes used to produce quantum codes, AG codes have received considerable attention. 

As an application of Lemma \ref{CSS} to AG codes, La Guardia and Pereira proposed in \cite{LGP} the following \emph{general $t$-point construction}.
\begin{lemma}{\rm{\cite[Theorem 3.1]{LGP}}}{\rm{(General $t$-point construction)}}\label{generaltpoint}
Let $F/\mathbb{F}_q$ be an algebraic function field of genus $g$ and with $n+t$ distinct points $\mathbb{F}_q$-rational for some $n,t>0$. For every $i=1,...,t$, let $a_i,b_i$ be positive integers such that $a_i\leq b_i$ and $$2g-2<\sum_{i=1}^t a_i<\sum_{i=1}^t b_i < n .$$ 
Then there exists a $[[n,k,d]]_q$-quantum code with $k=\displaystyle \sum_{i=1}^t b_i-\sum_{i=1}^t a_i$ and $$d\geq \min\lbrace n-\sum_{i=1}^t b_i, \sum_{i=1}^t a_i-(2g-2)\rbrace.$$
\end{lemma}

By applying Lemma \ref{generaltpoint} to the curve $\cC$ the following result is obtained.

\begin{proposition}\label{genquan}
Let $a,b\in \mathbb{N}$ such that $$\bar{q}(q-1)-2<a<b<q^2.$$
Then there exists a $[[q^2,b-a,d]]_{q}$ quantum code, where $$d\geq \min\lbrace q^2-b,\; a-\bar{q}(q-1)+2 \rbrace .$$
\end{proposition}

Many of the properties of AG codes that give rise to good quantum codes were captured in the definition of Castle
curves and weak Castle curves \cite{Castle2,Castle3}.

\begin{definition*}
Let $\mathcal{X}$ be a curve defined over $\mathbb{F}_q$ and $Q$ be an $\mathbb F_q$-rational place of $\mathcal X$. Then the pair $(\mathcal{X},Q)$ is called \textit{Castle} if the following conditions are satisfied.
\begin{itemize}
\item[$C1)$] The Weierstrass semigroup $H(Q)$ is symmetric.
\item[$C2)$] $\vert \mathcal{X}(\mathbb{F}_q)\vert=qm(H(Q))+1$.
\end{itemize}
\end{definition*}

All the Deligne-Lusztig curves are Castle. 

\begin{definition*}
Let $\mathcal{X}$ be a curve defined over $\mathbb{F}_q$ and $Q$ be an $\mathbb F_q$-rational place of $\mathcal X$. Then the pair $(\mathcal{X},Q)$ is called \textit{weak Castle} if the following conditions are satisfied.
\begin{itemize}
\item[$C1)$] The Weierstrass semigroup $H(Q)$ is symmetric;
\item[$WC2)$] For some integer $\ell$, there exists a morphism $f : \mathcal{X}\rightarrow \mathbb{P}^1 = \overline{\mathbb F}_q\cup\{\infty\}$ such that $(f)_{\infty}=\ell Q$ and there exists a set $U=\lbrace \alpha_1,...,\alpha_h \rbrace \subseteq \mathbb{F}_q$, such that for every $i=1,...,h$, $ f^{-1}(\alpha_i)\subseteq \mathcal{X}(\mathbb{F}_q)$ and $\vert f^{-1}(\alpha_i) \vert =\ell$.
\end{itemize}
\end{definition*}
Every Castle curve is weak Castle, since the rational function $f\in \mathcal{L}(Q)$ with $(f)_\infty = m(H(Q)) Q$ and $U = \mathbb{F}_q$ satisfy $WC2)$; see \cite[Proposition 2.5]{Castle2}.
If $(\mathcal{X},Q)$ is weak Castle, define
\begin{equation}\label{divD}
D=\sum_{i=1}^h \sum_{j=1}^\ell P_j^i,
\end{equation}
where $f^{-1}(\alpha_i)=\lbrace P_1^i,\ldots,P_\ell ^i\rbrace$ for every $i=1,\ldots,h$.

The one-point AG codes $C(D,rQ)$ are called \emph{Castle} or \emph{weak Castle codes}. Thanks to the weak Castle condition, these codes can be treated in an unified way. As it was proved, Castle and weak Castle curves provide families of codes with excellent parameters that satisfy certain self-orthogonality properties, making them good candidates for obtaining performing quantum stabilizer codes. 

\begin{proposition}{\rm{(\cite[Proposition 1, Proposition 2, and Corollary 2]{Castle3})}}\label{WeakProp}
Let $(\mathcal{X},Q)$ be a Castle curve of genus $g$ and $C(D,rQ)$ be a Castle code from $\mathcal{X}$.
Define $r^{\perp}=n+2g-2-r$, where $n$ is the length of $C(D,rQ)$.
Then the following properties hold:
\begin{itemize}
\item[(i)]  Let $f\in \mathcal{L}(Q)$ be a rational function such that $(f)_\infty = m(H(Q)) Q$. If $\text{div}(df)=(2g-2)Q$, then $C(D,rQ)^\perp = C(D,r^\perp Q)$.
\item[(ii)] The divisors $D$ and $rQ$ are equivalent. Also, for every $r<n$, $C(D,rQ)$ attains the designed minimum distance $d^*$ if and only if $C(D,(n-r)Q)$ attains the designed minimum distance as well.
\item[(iii)] $(2g-2)Q$ and $(n+2g-2)Q-D$ are canonical divisors, and there exists $x\in (\mathbb{F}_{q}^*)^n$ such that $C(D,rQ)^{\perp}=x\cdot C(D, r^{\perp}Q)$.
\item[(iv)] For every $i=1,...,r$, let $r_i:=\min \lbrace r: \ell (rQ)-\ell ((r-n)Q)\geq i \rbrace$ and $C_i:=C(D,r_iQ)$.
Then $C_i$ has dimension $i$, and
$$C_0=(0)\subset C_1\subset \cdots \subset C_n=\mathbb{F}_q^n$$ is a formally self-dual sequence of codes.
\item[(v)] If $2i\leq n$, then there exist quantum codes with parameters $[[n,n-2i,\geq d(C_{n-i})]]_{q}$ where $d(C_{n-i})\geq n-r_{n-i}+\gamma_{a+1}$, with $a=\ell((r_{n-i}-n)Q)$ and $$\gamma_{a+1}=\min \lbrace \deg(A) : A \textrm{ is a rational divisor on } \mathcal{X} \textrm{ with } \ell(A)\geq a+1\rbrace.$$
\end{itemize}
\end{proposition}

The following statement is a consequence of Section \ref{numberrationalpoints} and Proposition \ref{simmetrico}.

\begin{proposition}\label{èCastle}
The pair $(\mathcal{C},\mathcal{P}_\infty)$ is Castle.
\end{proposition}

Numerical computations seem to suggest that $(\mathcal{C},\mathcal{P}_\infty)$ is never Castle over $\mathbb F_{q^i}$ if $i>1$. However, being Castle over $\mathbb{F}_q$, it is readily seen that $(\mathcal{C},\mathcal{P}_\infty)$ is weak Castle over $\mathbb F_{q^i}$ for every $i\geq 1$. We provide an explicit proof of this fact.

\begin{proposition}\label{èweakCastle}
The pair $(\mathcal{C},\mathcal{P}_\infty)$ is weak Castle over $\mathbb{F}_{q^i}$, $i\geq 1$.
\end{proposition}
\begin{proof}
To prove the claim it is enough to show that there exists a function defined over $\mathbb F_{q^i}$ whose pole divisor is $\ell\mathcal P_\infty$, $\ell>0$, and such that its zeros are $\ell$ distinct $\mathbb F_{q^i}$-rational points of $\mathcal{C}$.
A possible choice is to consider $x\in \mathbb{F}_q(\mathcal{C)}$, since for any element $a\in \fq\subset \mathbb F_{q^i}$ the equation
$$
Y^q+Y=a^{q_0}(a^q+a)=0
$$
has $q$ distinct solutions in $\mathbb F_{q^i}$, and by Proposition \ref{scho} its pole divisor is $q\mathcal P_\infty$.
\end{proof}

Now we construct quantum codes from $\cC$ exploiting the Castle property of $(\mathcal{C},\mathcal{P}_\infty)$. Let $D$ be as in Equation $\eqref{divD}$, namely
$$
D=\sum_{P\in \mathcal{C}(\mathbb{F}_q)\setminus\{ \mathcal{P}_{\infty}\}} P.
$$
Then $C(D,r\mathcal{P}_\infty)$, $r>0$, are Castle codes of length $n=q^2$.
Moreover, with the notations of Proposition \ref{WeakProp}, since all the zeros of $x$ are simple and its unique pole $\mathcal{P}_{\infty}$ is totally ramified, we have  ${\text{div}}(dx)=(2g_\cC-2)\mathcal{P}_\infty$. Therefore, by (i) of Proposition \ref{WeakProp}, $C(D,r\mathcal{P}_\infty)^\perp = C(D,r^\perp \mathcal{P}_\infty)$.
Now, let
$$
H(\mathcal{P}_\infty)=\{\rho_0=0<\rho_1<\rho_2<\cdots\}.
$$
For $\rho_a, \rho_{a+b}\in H(\mathcal{P}_\infty)$, with $a,b\geq 1$ consider the codes 
\begin{equation*}
C_{a+b}:=C^{\perp}(D,\rho_{a+b}\mathcal{P}_\infty) \quad \text{ and } \quad C_a:=C^{\perp}(D,\rho_{a}\mathcal{P}_\infty),
\end{equation*} whose dimensions are $k_1=q^2-h_{a+b}$ and $k_2=q^2-h_{a}$, where $h_i$ is the number of non-gaps at $\mathcal{P}_\infty$ that do not exceed $i$. Note that $C_{a+b} \subseteq C_a$ and $k_2-k_1=b$. Then the CSS construction yields a $[[q^2,b,d]]_{q}$-quantum code such that
$d\geq \min \lbrace d_{ORD}(C_a),d_{1} \rbrace$, where $d_1$ is the minimum distance of the code $C(D,\rho_{a+b}\mathcal{P}_\infty)$. Since $C(D,\rho_{a+b}\mathcal{P}_\infty)=C^{\perp}(D,\rho_{a+b}^{\perp}\mathcal{P}_\infty)$, the lower bound on $d$ reads
\begin{equation}\label{stima}
d\geq\min\lbrace d_{ORD}(C_a), d_{ORD}(C^{\perp}(D,\rho_{a+b}^{\perp}P_{\infty})) \rbrace.
 \end{equation}
Note that the order bound can be computed only in terms of the Weierstrass semigroup $H(\mathcal{P}_\infty)$, that we determined explicitly in Section \ref{Section:semi}.

\section*{Acknowledgements}
The research of M. Timpanella was partially supported  by the Italian National Group for Algebraic and Geometric Structures and their Applications (GNSAGA - INdAM). The author is funded by the project ``Metodi matematici per la firma digitale ed il cloud computing" (Programma Operativo Nazionale (PON) ``Ricerca e Innovazione" 2014-2020, University of Perugia). The author would like to thank Massimo Giulietti for his helpful suggestions.

\end{document}